%% file: paper.tex
\documentclass{article} 

\input{preamble}

\PassOptionsToPackage{square, numbers, compress}{natbib}

\usepackage[nonatbib,final]{neurips_2019}

\usepackage[utf8]{inputenc} 
\usepackage[T1]{fontenc}    
\usepackage{hyperref}       
\usepackage{url}            
\usepackage{booktabs}       
\usepackage{amsfonts}       
\usepackage{nicefrac}       
\usepackage{microtype}      

\usepackage{bibentry}

\title{Efficiently escaping saddle points on manifolds}

\author{%
  Chris Criscitiello \\
  Department of Mathematics\\
  Princeton University\\
  Princeton, NJ 08544 \\
  \texttt{ccriscitiello6@gmail.com} \\
  \And 
  Nicolas Boumal \\
  Department of Mathematics\\
  Princeton University\\
  Princeton, NJ 08544 \\
  \texttt{nboumal@math.princeton.edu} \\
}

\begin{document}

\maketitle

\begin{abstract}
Smooth, non-convex optimization problems on Riemannian manifolds occur in machine learning as a result of orthonormality, rank or positivity constraints. First- and second-order necessary optimality conditions state that the Riemannian gradient must be zero, and the Riemannian Hessian must be positive semidefinite. Generalizing Jin et al.'s recent work on perturbed gradient descent (PGD) for optimization on linear spaces [How to Escape Saddle Points Efficiently (2017)~\cite{Jina2017}, Stochastic Gradient Descent Escapes Saddle Points Efficiently (2019)~\cite{Jin2019}], we propose a version of perturbed Riemannian gradient descent (PRGD) to show that necessary optimality conditions can be met approximately with high probability, without evaluating the Hessian. Specifically, for an arbitrary Riemannian manifold $\mathcal{M}$ of dimension $d$, a sufficiently smooth (possibly non-convex) objective function $f$, and under weak conditions on the retraction chosen to move on the manifold, with high probability, our version of PRGD produces a point with gradient smaller than $\epsilon$ and Hessian within $\sqrt{\epsilon}$ of being positive semidefinite in $O((\log{d})^4 / \epsilon^{2})$ gradient queries. This matches the complexity of PGD in the Euclidean case. Crucially, the dependence on dimension is low. This matters for large-scale applications including PCA and low-rank matrix completion, which both admit natural formulations on manifolds. The key technical idea is to generalize PRGD with a distinction between two types of gradient steps: ``steps on the manifold'' and ``perturbed steps in a tangent space of the manifold.'' Ultimately, this distinction makes it possible to extend Jin et al.'s analysis seamlessly.
\end{abstract}

\section{Introduction}\label{intro}
Machine learning has stimulated interest in obtaining global convergence rates in non-convex optimization. Consider a possibly non-convex objective function $f \colon \mathbb{R}^d \rightarrow \mathbb{R}$.  We want to solve
\begin{equation}\label{theproblem}
    \min_{x \in \mathbb{R}^d} f(x).
\end{equation}
This is hard in general. Instead, we usually settle for approximate first-order critical (or stationary) points where the gradient is small, or second-order critical (or stationary) points where the gradient is small and the Hessian is nearly positive semidefinite.

One of the simplest algorithms for solving \eqref{theproblem} is gradient descent (GD): given $x_0$, iterate 
\begin{equation}\label{GD}
    x_{t+1} = x_t - \eta \nabla f(x_t).
\end{equation}
It is well known that if $\nabla f$ is Lipschitz continuous, with appropriate step-size $\eta$, GD converges to first-order critical points. However, it may take exponential time to reach an approximate second-order critical point, thus, to escape \emph{saddle points}~\cite{du2017gradient}.
%
There is an increasing amount of evidence that saddle points are a serious obstacle to the practical success of local optimization algorithms such as GD~\cite{Pascanu2014, Ge2015}. 
This calls for algorithms which provably escape saddle points efficiently. We focus on methods which only have access to $f$ and $\nabla f$ (but not $\nabla^2 f$) through a black-box model.


Several methods add noise to GD iterates in order to escape saddle points faster, under the assumption that $f$ has $L$-Lipschitz continuous gradient and $\rho$-Lipschitz continuous Hessian.
In this setting, an \emph{$\epsilon$-second-order critical point} is a point $x$ satisfying $\norm{\nabla f(x)} \leq \epsilon$ and $\nabla^2 f(x) \succeq -\sqrt{\rho\epsilon}I$. Under the \emph{strict saddle assumption}, with $\epsilon$ small enough, such points are near (local) minimizers~\cite{Ge2015, Jina2017}.

In 2015, Ge et al.~\cite{Ge2015} gave a variant of stochastic gradient descent (SGD) which adds isotropic noise to iterates, 
showing it produces an $\epsilon$-second-order critical point with high probability in $O({\text{poly}(d)}/{\epsilon^4})$ stochastic gradient queries.
In 2017, Jin et al.~\cite{Jina2017} presented a variant of GD, perturbed gradient descent (PGD), which reduces this complexity to $O((\log{d})^4/{\epsilon}^2)$ full gradient queries. Recently, Jin et al.~\cite{Jin2019} simplified their own analysis of PGD, and extended it to stochastic gradient descent.

Jin et al.'s PGD~\cite[Alg.~4]{Jin2019} works as follows: 
If the gradient is large at iterate $x_t$, $\norm{\nabla f(x_t)} > \epsilon$, then perform a gradient descent step: $x_{t+1} = x_t - \eta \nabla f(x_t)$.  If the gradient is small at iterate $x_t$, $\norm{\nabla f(x_t)} \leq \epsilon$, perturb $x_t$ by $\eta \xi$, with $\xi$ sampled uniformly from a ball of fixed radius centered at zero.  Starting from this new point $x_t + \eta \xi$, perform $\mathscr{T}$ gradient descent steps, arriving at iterate $x_{t + \mathscr{T}}$.  From here, repeat this procedure starting at $x_{t + \mathscr{T}}$.  Crucially, Jin et al.~\cite{Jin2019} show that, if $x_t$ is not an $\epsilon$-second-order critical point,
then the function decreases enough from $x_t$ to $x_{t + \mathscr{T}}$ with high probability, leading to an escape.

In this paper we generalize PGD to optimization problems on manifolds, i.e., problems of the form
\begin{equation}
    \min_{x \in \calM} f(x)
\end{equation}
where $\calM$ is an arbitrary Riemannian manifold and $f \colon \calM \rightarrow \mathbb{R}$ is sufficiently smooth~\cite{AbsilBook}.  Optimization on manifolds notably occurs in machine learning (e.g., PCA~\cite{PCA1}, low-rank matrix completion~\cite{PCA2}), computer vision (e.g.,~\cite{vision}) and signal processing (e.g.,~\cite{signal})---see~\cite{apps} for more. 
See~\cite{saddle1} and~\cite{saddle2} for examples of the strict saddle property on manifolds.

Given $x\in\calM$, the (Riemannian) gradient of $f$ at $x$, $\grad f(x)$, is a vector in the tangent space at $x$, $\T_x\calM$.  To perform gradient descent on a manifold, we need a way to move on the manifold along the direction of the gradient at $x$.  This is provided by a \emph{retraction} $\Retr_x$: a smooth map from $\T_x\calM$ to $\calM$.  Riemannian gradient descent (RGD) performs steps on $\calM$ of the form
\begin{equation}\label{manGD}
    x_{t+1} = \Retr_{x_t}(-\eta \grad f(x_t)).
\end{equation}

For Euclidean space, $\calM = \reals^d$, the standard retraction is $\Retr_x(s) = x + s$, in which case~\eqref{manGD} reduces to~\eqref{GD}.  For the sphere embedded in Euclidean space, $\calM = S^{d} \subset \reals^{d+1}$, a natural retraction is given by metric projection to the sphere: $\Retr_x(s) = (x+s)/\norm{x+s}$.

For $x\in \calM$, define the \emph{pullback} $\hat{f}_x = f \circ \Retr_x \colon \T_x\calM \rightarrow \reals$, conveniently defined on a linear space.  If $\Retr$ is nice enough (details below), the Riemannian gradient and Hessian of $f$ at $x$ equal the (classical) gradient and Hessian of $\hat{f}_x$ at the origin of $\T_x\calM$.  Since $\T_x\calM$ is a vector space, if we perform GD on $\hat{f}_x$, we can almost directly apply Jin et al.'s analysis~\cite{Jin2019}.  This motivates the two-phase structure of our \emph{perturbed Riemannian gradient descent} (PRGD), listed as Algorithm~\ref{PRGD}.

Our PRGD is a variant of RGD \eqref{manGD} and a generalization of PGD.
It works as follows: If the gradient is large at iterate $x_t \in \calM$, $\norm{\grad f(x_t)} > \epsilon$, perform an RGD step: $x_{t+1} = \Retr_{x_t}({-\eta \grad f(x_t)})$.  We call this a ``step on the manifold.''  If the gradient at iterate $x_t$ is small, $\norm{\grad f(x_t)} \leq \epsilon$, then \emph{perturb in the tangent space} $\T_{x_t}\calM$. After this perturbation, execute at most $\mathscr{T}$ gradient descent steps \emph{on the pullback} $\hat{f}_{x_t}$, in the tangent space. We call these ``tangent space steps.''  We denote this sequence of $\mathscr{T}$ tangent space steps by $\{s_j\}_{j\geq 0}$.  This sequence of steps is performed by $\textproc{TangentSpaceSteps}$: a deterministic, vector-space procedure---see Algorithm~\ref{PRGD}.  

By distinguishing between gradient descent steps on the manifold and those in a tangent space, we can apply Jin et al.'s analysis almost directly~\cite{Jin2019}, allowing us to prove PRGD reaches an $\epsilon$-second-order critical point on $\calM$ in $O((\log{d})^4 / \epsilon^2)$ gradient queries. Regarding regularity of $f$, we require its pullbacks to satisfy Lipschitz-type conditions, as advocated in~\cite{trustMan,arcMan}. The analysis is far less technical than if one runs all steps on the manifold. We expect that this two-phase approach may prove useful for the generalization of other algorithms and analyses from the Euclidean to the Riemannian realm.



Recently, Sun and Fazel~\cite{Fazel2018} provided the first generalization of PGD to certain manifolds with a polylogarithmic complexity in the dimension, improving earlier results by Ge et al.~\cite[App.~B]{Ge2015} which had a polynomial complexity. Both of these works focus on submanifolds of a Euclidean space, with the algorithm in~\cite{Fazel2018} depending on the equality constraints chosen to describe this submanifold.

At the same time as the present paper, Sun et al.~\cite{Sun2019prgd} improved their analysis to cover any complete Riemannian manifold with bounded sectional curvature. In contrast to ours, their algorithm executes all steps on the manifold. Their analysis requires the retraction to be the Riemannian exponential map (i.e., geodesics). Our regularity assumptions are similar but different: while we assume Lipschitz-type conditions on the pullbacks in small balls around the origins of tangent spaces, Sun et al.\ make Lipschitz assumptions on the cost function directly, using parallel transport and Riemannian distance. As a result, curvature appears in their results. We make no explicit assumptions on $\calM$ regarding curvature or completeness, though these may be implicit in our regularity assumptions: see Section~\ref{sec:curvature}.

\begin{algorithm}
\caption{$\text{PRGD}(x_0, \eta, r, \mathscr{T}, \epsilon, T, b)$}\label{PRGD}
\begin{algorithmic}[1]
\State $t \gets 0$
\While {$t \leq T$}
\If {$\norm{\grad f(x_t)} > \epsilon$}
\State $x_{t+1} \gets \textproc{TangentSpaceSteps}(x_t, 0, \eta, b, 1)$ \Comment{Riemannian gradient descent step}
\State $t \gets t+1$
\Else
\State $\xi \sim \text{Uniform}(B_{x_t, r}(0))$ \Comment{perturb}
\State $s_0 = \eta \xi$
\State $x_{t+\mathscr{T}} \gets \textproc{TangentSpaceSteps}(x_t, s_0, \eta, b, \mathscr{T})$ \Comment{perform $\mathscr{T}$ steps in $\T_{x_t}\calM$}
\State $t \gets t + \mathscr{T}$
\EndIf
\EndWhile
\\
\Procedure{$\textproc{TangentSpaceSteps}(x, s_0, \eta, b, \mathscr{T})$}{}
\For {$j = 0, 1, \ldots, \mathscr{T} - 1$}
\State $s_{j+1} \gets s_j - \eta \nabla \hat{f}_{x}(s_j)$
\If {$\norm{s_{j+1}} \geq b$} \Comment{if the iterate leaves the interior of the ball $B_{x,b}(0)$}
\State $s_{\mathscr{T}} \gets s_j - \alpha \eta \nabla \hat{f}_{x}(s_j)$, where $\alpha \in (0,1]$ and $\norm{s_j - \alpha \eta \nabla \hat{f}_{x}(s_j)}=b$.
\State \textbf{break}
\EndIf
\EndFor
\State \textbf{return} $\Retr_x{(s_{\mathscr{T}})}$
\EndProcedure
\end{algorithmic}
\end{algorithm}

\subsection{Main result}
Here we state our result informally. Formal results are stated in subsequent sections.


\begin{theorem}[Informal]
Let $\calM$ be a Riemannian manifold of dimension $d$ equipped with a retraction $\Retr$.  Assume $f \colon \calM \rightarrow \mathbb{R}$ is twice continuously differentiable, and furthermore:
\begin{enumerate}
    \item[A1.] $f$ is lower bounded. 
    \item[A2.] The gradients of the pullbacks $f \circ \Retr_x$ uniformly satisfy a Lipschitz-type condition.
    \item[A3.] The Hessians of the pullbacks $f \circ \Retr_x$ uniformly satisfy a Lipschitz-type condition.
    \item[A4.] The retraction $\Retr$ uniformly satisfies a second-order condition.
\end{enumerate}
Then, setting $T = O((\log{d})^4/\epsilon^2)$, PRGD visits several points with gradient smaller than $\epsilon$ and, with high probability, at least two-thirds of those points are $\epsilon$-second-order critical (Definition~\ref{def:epssocpmanifold}).
\end{theorem}

PRGD uses $O((\log{d})^4/\epsilon^2)$ gradient queries, and crucially no Hessian queries.
The algorithm requires
knowledge of the Lipschitz constants defined below, which makes this a mostly theoretical algorithm---but see Appendix~\ref{sec:PCA} for explicit constants in the case of PCA.

\subsection{Other related work}

Algorithms which efficiently escape saddle points can be classified into two families: first-order and second-order methods.  First-order methods only use function value and gradient information. SGD and PGD are first-order methods. Second-order methods also access Hessian information. Newton's method, trust regions~\cite{trust, trustMan} and adaptive cubic regularization~\cite{arc, arcMan, zhang2018cubicregmanifold} are second-order methods.

As noted above, Ge et al.~\cite{Ge2015} and Jin et al.~\cite{Jina2017} escape saddle points (in Euclidean space) by exploiting noise in iterations.  There has also been similar work for normalized gradient descent~\cite{Levy2016}.  Expanding on~\cite{Jina2017}, Jin et al.~\cite{Jinb2017} give an accelerated PGD algorithm (PAGD) which reaches an $\epsilon$-second-order critical point of a non-convex function $f$ with high probability in $O({(\log{d})^6}/{\epsilon^{7/4}})$ iterations.  In~\cite{Jin2019}, Jin et al.\ show that a stochastic version of PGD reaches an $\epsilon$-second-order critical point in $O(d / \epsilon^4)$ stochastic gradient queries; only $O(\text{poly}(\log{d}) / \epsilon^4)$ queries are needed if the stochastic gradients are well behaved. For an analysis of PGD under convex constraints, see~\cite{mokhtari2018escaping}.

There is another line of research, inspired by Langevin dynamics, in which judiciously scaled Gaussian noise is added at every iteration.  We note that although this differs from the first incarnation of PGD in~\cite{Jina2017}, this resembles a simplified version of PGD in~\cite{Jin2019}.  Sang and Liu~\cite{Sang2018} develop an algorithm (adaptive stochastic gradient Langevin dynamics, ASGLD), which provably reaches an $\epsilon$-second-order critical point in $O({\log{d}}/{\epsilon^4})$ with high probability.  With full gradients, AGSLD reaches an $\epsilon$-second-order critical point in $O({\log{d}}/{\epsilon^2})$ queries with high probability.

One might hope that the noise inherent in vanilla SGD would help it escape saddle points without noise injection.  Daneshmand et al.~\cite{Daneshmand2018} propose the correlated negative curvature assumption (CNC), under which they prove that SGD reaches an $\epsilon$-second-order critical point in $O({\epsilon^{-5}})$ queries with high probability. They also show that, under the CNC assumption, a variant of GD (in which iterates are perturbed only by SGD steps) efficiently escapes saddle points.  Importantly, these guarantees are completely dimension-free.

A first-order method can include approximations of the Hessian (e.g., with a difference of gradients). For example, Allen-Zhu's Natasha 2 algorithm~\cite{AllenZhua2017} uses first-order information (function value and stochastic gradients) to search for directions of negative curvature of the Hessian.  Natasha 2 reaches an $\epsilon$-second-order critical point in $O({\epsilon^{-13/4}})$ iterations.

Many classical optimization algorithms have been generalized to optimization on manifolds, including gradient descent, Newton's method, trust regions and adaptive cubic regularization~\cite{edelman1998geometry,AbsilBook,genrtr,newton,trustMan,arcMan,bento2017iterationcomplexity,zhang2018cubicregmanifold}.  Bonnabel~\cite{bonnabel} extends stochastic gradient descent to Riemannian manifolds and proves that Riemannian SGD converges to critical points of the cost function.  Zhang et al.~\cite{speedup} and Sato et al.~\cite{speedup2} both use variance reduction to speed up SGD on Riemannian manifolds.

\section{Preliminaries: Optimization on manifolds}\label{prelims}
We review the key definitions and tools for optimization on manifolds. For more information, see~\cite{AbsilBook}.  Let $\calM$ be a $d$-dimensional Riemannian manifold: a real, smooth $d$-manifold equipped with a Riemannian metric.  We associate with each $x \in \calM$ a $d$-dimensional real vector space $\T_x\calM$, called the tangent space at $x$.  For embedded submanifolds of $\reals^n$, we often visualize the tangent space as being tangent to the manifold at $x$.  The Riemannian metric defines an inner product $\inner{\cdot}{\cdot}_x$ on the tangent space $\T_x\calM$, with associated norm $\norm{\cdot}_x$.  We denote these by $\inner{\cdot}{\cdot}$ and $\norm{\cdot}$ when $x$ is clear from context.  A vector in the tangent space is a tangent vector.  The set of pairs $(x, s_x)$ for $x\in\calM, s_x\in\T_x\calM$ is called the tangent bundle $\T\calM$.  Define $B_{x,r}(s) = \{\Dot{s}\in\T_x\calM : \norm{\Dot{s}-s}_x \leq r\}$: the closed ball of radius $r$ centered at $s \in \T_x\calM$.  We occasionally denote $B_{x,r}(s)$ by $B_r(s)$ when $x$ is clear from context.  Let $\text{Uniform}(B_{x,r}(s))$ denote the uniform distribution over the ball $B_{x,r}(s)$.

The Riemannian gradient $\grad\!f(x)$ of a differentiable function $f$ at $x\in\calM$ is the unique vector in $\T_x\calM$ satisfying $\D f(x)[s] = \inner{\grad f(x)}{s}_x$ $\forall s \in \T_x\calM$, where $\D f(x)[s]$ is the directional derivative of $f$ at $x$ along $s$.  The Riemannian metric gives rise to a well-defined notion of derivative of vector fields called the Riemannian (or Levi--Civita) connection $\nabla$.  The Hessian of $f$ is the derivative of the gradient vector field: $\Hess\!f(x)[u] = \nabla_{u}\grad\!f (x)$.  The Hessian describes how the gradient changes.  $\Hess\!f(x)$ is a symmetric linear operator on $\T_x\calM$.  If the manifold is a Euclidean space, $\calM = \reals^d$, with the standard metric $\inner{x}{y} = x^{T}y$, the Riemannian gradient $\grad\!f$ and Hessian $\Hess\!f$ coincide with the standard gradient $\nabla f$ and Hessian $\nabla^2 f$ (mind the overloaded notation $\nabla$). 

As discussed in Section \ref{intro}, the retraction is a mapping which allows us to move along the manifold from a point $x$ in the direction of a tangent vector $s \in \T_x\calM$.  Formally:
\begin{definition}[Retraction, from~\cite{AbsilBook}] \label{def:retr}
A retraction on a manifold $\calM$ is a smooth mapping $\Retr$ from the tangent bundle $\T\calM$ to $\calM$ satisfying properties 1 and 2 below.  Let $\Retr_x \colon \T_x\calM \rightarrow \calM$ denote the restriction of $\Retr$ to $\T_x\calM$.
\begin{enumerate}
    \item $\Retr_x(0_x) = x$, where $0_x$ is the zero vector in $\T_x\calM$.
    \item The differential of $\Retr_x$ at $0_x$, $\D\Retr_x(0_x)$, is the identity map.
\end{enumerate}
\end{definition}
(Our algorithm and theory only require $\Retr$ to be defined in balls of a fixed radius around the origins of tangent spaces.)
Recall these special retractions, which are good to keep in mind for intuition: on $\calM = \Rd$, we typically use $\Retr_x(s) = x+s$, and on the unit sphere we typically use $\Retr_x(s) = (x+s)/\norm{x+s}$.

For $x$ in $\calM$, define the pullback of $f$ from the manifold to the tangent space by
\begin{align*}
	\hat{f}_x = f \circ \Retr_x \colon \T_x\calM \rightarrow \reals.
\end{align*}
This is a real function on a vector space. Furthermore, for $x \in \calM$ and $s \in \T_x\calM$, let
\begin{align*}
	T_{x,s} = \D\Retr_x(s) \colon \T_x\calM \to \T_{\Retr_x(s)}\calM
\end{align*}
denote the differential of $\Retr_x$ at $s$ (a linear operator).
The gradient and Hessian of the pullback admit the following nice expressions in terms of those of $f$, and the retraction.
\begin{lemma}[Lemma 5.2 of~\cite{arcMan}]\label{lems}
For $f\colon\calM\rightarrow\reals$ twice continuously differentiable, $x\in\calM$ and $s\in\T_x\calM$, with $T_{x,s}^*$ denoting the adjoint of $T_{x,s}$,
\begin{align}
    \nabla \hat{f}_x(s) & = T_{x,s}^*\grad f(\Retr_x(s)), & 
    \nabla^2 \hat{f}_x(s) & = T_{x,s}^* \Hess f(\Retr_x(s)) T_{x,s} + W_s,
\end{align}
where $W_s$ is a symmetric linear operator on $\T_x\calM$ defined through polarization by
\begin{equation}
    \inner{W_s[\Dot{s}]}{\Dot{s}} = \inner{\grad f(\Retr_x(s))}{\gamma''(0)},
\end{equation}
with $\gamma''(0) \in \T_{\Retr_x(s)}\calM$ the intrinsic acceleration on $\calM$ of $\gamma(t) = \Retr_x(s + t\Dot{s})$ at $t=0$.
\end{lemma}
The velocity of a curve $\gamma\colon \reals \rightarrow \calM$ is $\frac{d \gamma}{d t} = \gamma'(t)$.  The intrinsic acceleration $\gamma''$ of $\gamma$ is the covariant derivative (induced by the Levi--Civita connection) of the velocity of $\gamma$: $\gamma'' = \frac{\D}{d t} \gamma'$.  When $\calM$ is a Riemannian submanifold of $\reals^n$, $\gamma''(t)$ does not necessarily coincide with $\frac{d^2 \gamma}{d t^2}$: in this case, $\gamma''(t)$ is the orthogonal projection of $\frac{d^2 \gamma}{d t^2}$ onto $\T_{\gamma(t)} \calM$.


\section{PRGD efficiently escapes saddle points}
We now precisely state the assumptions, the main result, and some important parts of the proof of the main result, including the main obstacles faced in generalizing PGD to manifolds.  A full proof of all results is provided in the appendix.

\subsection{Assumptions}\label{assumptions}
The first assumption, namely, that $f$ is lower bounded, ensures that there are points on the manifold where the gradient is arbitrarily small.

\begin{assumption}\label{as:A1}
$f$ is lower bounded: $f(x) \geq f^*$ for all $x \in \calM$.
\end{assumption}

Generalizing from the Euclidean case, we assume Lipschitz-type conditions on the gradients and Hessians of the pullbacks $\hat f_x = f \circ \Retr_x$.  For the special case of $\calM = \mathbb{R}^d$ and $\Retr_x(s) = x + s$, these assumptions hold if the gradient $\nabla f(\cdot)$ and Hessian $\nabla^2 f(\cdot)$ are each Lipschitz continuous, as in~\cite[A1]{Jin2019} (with the same constants).  The Lipschitz-type assumptions below are similar to
assumption A2 of~\cite{arcMan}. Notice that these assumptions involve both the cost function and the retraction: this dependency is further discussed in~\cite{trustMan,arcMan} for a similar setting.

\begin{assumption}\label{as:A2}
There exist $b_1 > 0$ and $L > 0$ such that $\forall x \in \calM$ and $\forall s \in \T_x \calM$ with $\norm{s} \leq b_1$,
$$\norm{\nabla\hat{f}_x(s) - \nabla\hat{f}_x(0)} \leq L\norm{s}.$$
\end{assumption}

\begin{assumption}\label{as:A3}
There exist $b_2 > 0$ and $\rho > 0$ such that $\forall x \in \calM$ and $\forall s \in \T_x \calM$ with $\norm{s} \leq b_2$,
$$\norm{\nabla^2\hat{f}_x(s) - \nabla^2\hat{f}_x(0)} \leq \rho\norm{s},$$
where on the left-hand side we use the operator norm.
\end{assumption}

More precisely, we only need these assumptions to hold at the iterates $x_0, x_1, \ldots$
Let $b = \min\{b_1, b_2\}$
(to reduce the number of parameters in Algorithm~\ref{PRGD}).
The next assumption requires the chosen retraction to be well behaved, in the sense that the (intrinsic) acceleration of curves $\gamma_{x,s}$ on the manifold, defined below, must remain bounded---compare with Lemma~\ref{lems}.
\begin{assumption}\label{as:A4}
	There exists $\beta \geq 0$ such that, for all $x \in \calM$ and $s \in \T_x\calM$ satisfying $\norm{s}=1$, the curve $\gamma_{x,s}(t) = \Retr_x(t s)$ has initial acceleration bounded by $\beta$: $\norm{\gamma_{x,s}''(0)} \leq \beta$.
\end{assumption}
If Assumption \ref{as:A4} holds with $\beta=0$, $\Retr$ is said to be \emph{second order}~\cite[p107]{AbsilBook}. Second-order retractions include the so-called exponential map and the standard retractions on $\reals^d$ and the unit sphere mentioned earlier---see~\cite{malick} for a large class of such retractions on relevant manifolds.

\begin{definition} \label{def:epssocpmanifold}
A point $x \in \calM$ is an $\epsilon$-second-order critical point of the twice-differentiable function $f \colon \calM \rightarrow \mathbb{R}$ satisfying Assumption~\ref{as:A3} if
\begin{align}
	\norm{\grad f(x)} & \leq \epsilon, & &  \text{ and } & \lambda_{\min}(\Hess f(x)) & \geq -\sqrt{\rho \epsilon},
\end{align}
where $\lambda_{\min}(H)$ denotes the smallest eigenvalue of the symmetric operator $H$.
\end{definition}


For compact manifolds, all of these assumptions hold (all proofs are in the appendix):
\begin{lemma} \label{lem1}
Let $\calM$ be a compact Riemannian manifold equipped with a retraction $\Retr$.  Assume $f\colon\calM\rightarrow\reals$ is three times continuously differentiable. Pick an arbitrary $b > 0$.  Then, there exist $f^*, L > 0, \rho>0$ and $\beta \geq 0$ such that Assumptions \ref{as:A1}, \ref{as:A2}, \ref{as:A3} and \ref{as:A4} are satisfied.
\end{lemma}

\subsection{Main results}\label{mainresults}

Recall that PRGD (Algorithm~\ref{PRGD}) works as follows.  If $\norm{\grad f(x_t)} > \epsilon$, perform a Riemannian gradient descent step, $x_{t+1} = \Retr_{x_t}({-\eta \grad f(x_t)})$.  If $\norm{\grad f(x_t)} \leq \epsilon$, then perturb, i.e., sample $\xi \sim \text{Uniform}(B_{x_t, r}(0))$ and let $s_0 = \eta \xi$. After this perturbation, remain in the tangent space $\T_{x_t}\calM$ and do (at most) $\mathscr{T}$ gradient descent steps on the pullback $\hat{f}_{x_t}$, starting from $s_0$.  We denote this sequence of $\mathscr{T}$ tangent space steps by $\{s_j\}_{j\geq 0}$.  This sequence of gradient descent steps is performed by $\textproc{TangentSpaceSteps}$: a deterministic procedure in the (linear) tangent space. 

One difficulty with this approach is that, under our assumptions, for some $x = x_t$, $\nabla\hat{f}_{x}$ may not be Lipschitz continuous in all of $\T_x\calM$.  However, it is easy to show that $\nabla\hat{f}_{x}$ is Lipschitz continuous in the ball of radius $b$ by compactness, uniformly in $x$. This is why we limit our algorithm to these balls. If the sequence of iterates $\{s_j\}_{j\geq 0}$ escapes the ball $B_{x,b}(0) \subset \T_{x}\calM$ for some $s_j$, $\textproc{TangentSpaceSteps}$ returns the point between $s_{j-1}$ and $s_j$ on the boundary of that ball.


Following~\cite{Jin2019}, we use a set of carefully balanced parameters. 
Parameters $\epsilon$ and $\delta$ are user defined.  The claim in Theorem~\ref{thmMaster} below holds with probability at least $1-\delta$.
Assumption~\ref{as:A1} provides parameter $f^*$.
Assumptions~\ref{as:A2} and~\ref{as:A3} provide parameters $L, \rho$ and $b = \min\{b_1, b_2\}$.
As announced, the latter two assumptions further ensure Lipschitz continuity of the gradients of the pullbacks in balls of the tangent spaces, uniformly: this defines the parameter $\ell$, as prescribed below.
\begin{lemma} \label{lem:Aell}
	Under Assumptions~\ref{as:A2} and~\ref{as:A3}, there exists $\ell \in [L, L+\rho b]$ such that, for all $x \in \mathcal{M}$, the gradient of the pullback, $\nabla \hat f_x$, is $\ell$-Lipschitz continuous in the ball $B_{x, b}(0)$.	
\end{lemma}


Then, choose $\chi > 1/4$ (preferably small) such that
\begin{align}
	\chi & \geq 4 \log_2\left(2^{31}\frac{\ell^2\sqrt{d}(f(x_0)-f^*)}{\delta\sqrt{\rho}\epsilon^{5/2}}\right), 
	\label{eq:chi_ineq}
\end{align}
and set algorithm parameters
\begin{align}
	\eta & = \frac{1}{\ell}, &
	r & = \frac{\epsilon}{400 \chi^3}, &
	\mathscr{T} & = \frac{\ell\chi}{\sqrt{\rho \epsilon}}, &
	\label{params} 
\end{align}
where $\chi$ is such that $\mathscr{T}$ is an integer.  We also use this notation in the proofs:
\begin{align} \label{eq:FLscr}
	\mathscr{F} & = \frac{1}{50 \chi^3}\sqrt{\frac{\epsilon^3}{\rho}}, & \mathscr{L} & = \frac{1}{4\chi}\sqrt{\frac{\epsilon}{\rho}}.
\end{align}
%


\begin{theorem} \label{thmMaster}
Assume $f$ satisfies Assumptions \ref{as:A1}, \ref{as:A2} and \ref{as:A3}.  For any $x_0 \in \calM$, with $0 < \epsilon \leq b^2\rho$, 
$L \geq \sqrt{\rho\epsilon}$, $\epsilon^{3/2} \leq 3\sqrt{\rho} \left( f(x_0) - f^* \right)$ and $ \delta \in (0,1)$, choose $\eta, r, \mathscr{T}$ as in~\eqref{params}.  Then, setting  
\begin{align}\label{eqnT}
T = 8\max\left\{\frac{\mathscr{T}}{3}, \frac{(f(x_0)-f^*)\mathscr{T}}{\mathscr{F}}, \frac{f(x_0)-f^*}{\eta \epsilon^2}\right\} = O\bigg(\frac{\ell(f(x_0) - f^*)}{\epsilon^2} (\log{d})^4\bigg),
\end{align}
$PRGD(x_0, \eta, r, \mathscr{T}, \epsilon, T, b)$ visits at least two iterates $x_t \in \calM$ satisfying $\norm{\grad f(x_t)} \leq \epsilon$. With probability at least $1-\delta$, at least two-thirds of those iterates satisfy
\begin{align*}
	\norm{\grad f(x_t)} & \leq \epsilon & & \text{ and } & \lambda_{\min}(\nabla^2 \hat{f}_{x_t}(0))  & \geq -\sqrt{\rho \epsilon}.
\end{align*}
The algorithm uses at most $T + \mathscr{T} \leq 2 T$ gradient queries (and no function or Hessian queries).
\end{theorem}
%

By Assumption \ref{as:A4} and Lemma \ref{lems}, $\nabla^2 \hat{f}_{x_t}(0)$ is close to $\Hess f(x_t)$, which allows us to conclude:

\begin{corollary}\label{thmMR}
Assume $f$ satisfies Assumptions~\ref{as:A1}, \ref{as:A2}, \ref{as:A3} and~\ref{as:A4}.  For an arbitrary $x_0 \in \calM$, with $0 < \epsilon \leq \min\{\rho / \beta^2,b^2\rho\}$,
$L \geq \sqrt{\rho\epsilon}$, $\epsilon^{3/2} \leq 3\sqrt{\rho} \left( f(x_0) - f^* \right)$ and $\delta \in (0,1)$, choose $\eta, r, \mathscr{T}$ as in~\eqref{params}.  Then, setting $T$ as in~\eqref{eqnT}, $PRGD(x_0, \eta, r, \mathscr{T}, \epsilon, T, b)$ visits at least two iterates $x_t \in \calM$ satisfying $\norm{\grad f(x_t)}\leq\epsilon$. With probability at least $1-\delta$, at least two-thirds of those iterates are $(4\epsilon)$-second-order points. If $\beta = 0$ (that is, the retraction is second order), then the same claim holds for $\epsilon$-second-order points instead of $4\epsilon$. The algorithm uses at most $T + \mathscr{T} \leq 2 T$ gradient queries.
\end{corollary}


Assume $\calM = \reals^d$ with standard inner product and standard retraction $\Retr_x(s) = x+s$.  As in~\cite{Jin2019}, assume $f\colon\reals^d \rightarrow\reals$ is lower bounded, $\nabla f$ is $L$-Lipschitz in $\reals^d$, and $\nabla^2 f$ is $\rho$-Lipschitz in $\reals^d$.  Then, Assumptions \ref{as:A1}, \ref{as:A2} and \ref{as:A3} hold with $b = +\infty$. Furthermore, Assumption~\ref{as:A4} holds with $\beta = 0$ so that $\nabla^2\hat{f}_x(0)  = \Hess f(x) = \nabla^2 f(x)$ (Lemma~\ref{lems}).  For all $x\in\calM$, $\nabla \hat{f}_x(s)$ has Lipschitz constant $\ell = L$ since $\hat{f}_x(s) = f(x+s)$.  Therefore, using $b = +\infty$, $\ell = L$ and choosing $\eta, r, \mathscr{T}$ as in~\eqref{params}, PRGD reduces to PGD, and Theorem \ref{thmMaster} recovers the result of Jin et al.~\cite{Jin2019}: this confirms that the present result is a bona fide generalization.

For the important special case of compact manifolds, Lemmas~\ref{lem1} and~\ref{lem:Aell} yield:
\begin{corollary}\label{compactCor}
Assume $\calM$ is a compact Riemannian manifold equipped with a retraction $\Retr$, and $f\colon\calM\rightarrow\reals$ is three times continuously differentiable.  Pick an arbitrary $b>0$. Then, Assumptions~\ref{as:A1}, \ref{as:A2}, \ref{as:A3}, \ref{as:A4} hold for some $L>0, \rho>0$, $\beta \geq 0$, so that Corollary~\ref{thmMR} applies with some $\ell \in [L, L+\rho b]$.
\end{corollary}

\begin{remark}
PRGD, like PGD (Algorithm 4 in~\cite{Jin2019}), does not specify which iterate is an $\epsilon$-second-order critical point.  However, it is straightforward to include a termination condition in PRGD which halts the algorithm and returns a suspected $\epsilon$-second-order critical point. Indeed, Jin et al.\ include such a termination condition in their original PGD algorithm~\cite{Jina2017}, which here would go as follows:  After performing a perturbation and $\mathscr{T}$ (tangent space) steps in $\T_{x_t}\calM$, return $x_t$ if $\hat{f}_{x_t}(s_{\mathscr{T}}) - \hat{f}_{x_t}(0) > -f_{\mathrm{thres}}$, i.e., the function value does not decrease enough.  The termination condition requires a threshold $f_{\mathrm{thres}}$ which is balanced like the other parameters of PRGD in~\eqref{params}.
\end{remark}


\subsection{Main proof ideas}
Theorem \ref{thmMaster} follows from the following two lemmas which we prove in the appendix.  These lemmas state that, in each round of the while-loop in PRGD, if $x_t$ is not at an $\epsilon$-second-order critical point, PRGD makes progress, that is, decreases the cost function value (the first lemma is deterministic, the second one is probabilistic).  Yet, the value of $f$ on the iterates can only decrease so much because $f$ is bounded below by $f^*$.  Therefore, the probability that PRGD does not visit an $\epsilon$-second-order critical point is low.
\begin{lemma}\label{lemDog}
Under Assumptions~\ref{as:A2} and~\ref{as:A3}, set $\eta = 1/\ell$ for some $\ell \geq L$. If $x\in\calM$ satisfies $\norm{\grad f(x)} > \epsilon$ with $\epsilon \leq b^2 \rho$ and $L \geq \sqrt{\rho\epsilon}$, then,
$$f(\textproc{TangentSpaceSteps}(x, 0, \eta, b, 1)) - f(x) \leq -\eta \epsilon^2/2.$$
\end{lemma}
\begin{lemma}\label{lemCat}
Under Assumptions~\ref{as:A2} and~\ref{as:A3}, let $x\in\calM$ satisfy both $\norm{\grad f(x)} \leq \epsilon$ and $\lambda_{\min}(\nabla^2 \hat{f}_x(0)) \leq -\sqrt{\rho\epsilon}$ with $\epsilon \leq b^2\rho$ and $L \geq \sqrt{\rho\epsilon}$.  Set $\eta, r, \mathscr{T}, \mathscr{F}$ as in \eqref{params} and \eqref{eq:FLscr}.  Let $s_0 = \eta \xi$ with $\xi \sim \text{Uniform}(B_{x,r}(0))$.  Then,
$$\mathbb{P}\big[f(\textproc{TangentSpaceSteps}(x, s_0, \eta, b, \mathscr{T})) - f(x) \leq -\mathscr{F}/2\big] \geq 1 - \frac{\ell\sqrt{d}}{\sqrt{\rho\epsilon}}2^{10-\chi/2}.$$
\end{lemma}

Lemma~\ref{lemDog} states that we are guaranteed to make progress if the gradient is large.  This follows from the sufficient decrease of RGD steps.  Lemma \ref{lemCat} states that, with perturbation, GD on the pullback escapes a saddle point with high probability. Lemma~\ref{lemCat} is analogous to Lemma 11 in~\cite{Jin2019}.

Let $\mathcal{X}_{\mathrm{stuck}}$ be the set of tangent vectors $s_0$ in $B_{x,\eta r}(0)$ for which GD on the pullback starting from $s_0$ does not escape the saddle point, i.e., the function value does not decrease enough after $\mathscr{T}$ iterations.  Following Jin et al.'s analysis~\cite{Jin2019}, we bound the width of this ``stuck region'' (in the direction of the eigenvector $e_1$ associated with the minimum eigenvalue of the Hessian of the pullback, $\nabla^2\hat{f}_x(0)$).  Like Jin et al., we do this with a coupling argument, showing that given two GD sequences with starting points sufficiently far apart, one of these sequences must escape.  This is formalized in Lemma \ref{lem:lemma_13} of the appendix.  A crucial observation to prove Lemma \ref{lem:lemma_13} is that, if the function value of GD iterates does not decrease much, then these iterates must be localized; this is formalized in Lemma~\ref{lem:lemma_12} of the appendix, which Jin et al.\ call ``improve or localize.''

We stress that the stuck region concept, coupling argument, improve or local paradigm, and details of the analysis are due to Jin et al.~\cite{Jin2019}: our main contribution is to show a clean way to generalize the algorithm to manifolds in such a way that the analysis extends with little friction. We believe that the general idea of separating iterations between the manifold and the tangent spaces to achieve different objectives may prove useful to generalize other algorithms as well.

\section{About the role of curvature of the manifold} \label{sec:curvature}

As pointed out in the introduction, concurrently with our work, Sun et al.~\cite{Sun2019prgd} have proposed another generalization of PGD to manifolds. Their algorithm executes all steps on the manifold directly (as opposed to our own, which makes certain steps in the tangent spaces), and moves around the manifold using the exponential map. To carry out their analysis, Sun et al.\ assume $f$ is regular in the following way. The Riemannian gradient is Lipschitz continuous in a Riemannian sense, namely,
\begin{align*}
	\forall x, y \in \calM, && \|\grad f(y) - \Gamma_x^y \grad f(x)\| & \leq L \dist(x, y),
\end{align*}
where $\Gamma_x^y \colon \T_x\calM \to \T_y\calM$ denotes parallel transport from $x$ to $y$ along any minimizing geodesic, and $\dist$ is the Riemannian distance. These notions are well defined if $\calM$ is a connected, complete manifold. Similarly, they assume the Riemannian Hessian of $f$ is Lipschitz continuous in a Riemannian sense:
\begin{align*}
	\forall x, y \in \calM, && \| \Hess f(y) - \Gamma_x^y \circ \Hess f(x) \circ \Gamma_y^x \| & \leq \rho \dist(x, y),
\end{align*}
in the operator norm. Using (and improving) sophisticated inequalities from Riemannian geometry, they map the perturbed sequences back to tangent spaces for analysis, where they run an adapted version of Jin et al.'s argument. In so doing, it appears to be crucial to use the exponential map, owing to its favorable interplay with parallel transport along geodesics and Riemannian distance, providing a good fit with the regularity conditions above.

As they map sequences back from the manifold to a common tangent space through the inverse of the exponential map, the Riemannian curvature of the manifold comes into play. Consequently, they assume $\calM$ has bounded sectional curvature (both from below and from above), and these bounds on curvature come up in their final complexity result: constants degrade if the manifold is more curved. 

Since Riemannian curvature does not occur in our own complexity result for PRGD, it is legitimate to ask: is curvature supposed to occur? If so, it must be hidden in our analysis, for example in the regularity assumptions we make, which are expressed in terms of pullbacks rather than with parallel transports. And indeed, in several attempts to deduce our own assumptions from those of Sun et al., invariably, we had to degrade $L$ and $\rho$ as a function of curvature---minding that these are only bounds. On the other hand, under the assumptions of Sun et al., one can deduce that the regularity assumptions required in~\cite{trustMan,arcMan} for the analysis of Riemannian gradient descent, trust regions and adaptive regularization by cubics hold with the exponential map, leading to curvature-free complexity bounds for all three algorithms. Thus, it is not clear that curvature should occur.

We believe this poses an interesting question regarding the complexity of optimization on manifolds: to what extent should it be influenced by curvature of the manifold? We intend to study this.

\section{Perspectives}
To perform PGD (Algorithm 4 of~\cite{Jin2019}), one must know the step size $\eta$, perturbation radius $r$ and the number of steps $\mathscr{T}$ to perform after perturbation.  These parameters are carefully balanced, and their values depend on the smoothness parameters $L$ and $\rho$.  In most situations, we do not know $L$ or $\rho$ (though see Appendix~\ref{sec:PCA} for PCA). An algorithm which does not require knowledge of $L$ or $\rho$ but still has the same guarantees as PGD would be useful. However, that certain regularity parameters must be known seems inevitable, in particular for the Hessian's $\rho$. Indeed, the main theorems make statements about the spectrum of the Hessian, yet the algorithm is not allowed to query the Hessian.

GD equipped with a backtracking line-search method achieves an $\epsilon$-first-order critical point in $O(\epsilon^{-2})$ gradient queries without knowledge of the Lipschitz constant $L$.  At each iterate $x_t$ of GD, backtracking line-search essentially uses function and gradient queries to estimate the gradient Lipschitz parameter near $x_t$.  Perhaps PGD can perform some kind of line-search to locally estimate $L$ and $\rho$.  We note that if $\rho$ is known and we use line-search-type methods to estimate $L$, there are still difficulties applying Jin et al.'s coupling argument.

Jin et al.~\cite{Jin2019} develop a stochastic version of PGD known as PSGD.  Instead of perturbing when the gradient is small and performing $\mathscr{T}$ GD steps, PSGD simply performs a stochastic gradient step and perturbation at each step.  Distinguishing between manifold steps and tangent space steps, we suspect it is possible to develop a Riemannian version of perturbed stochastic gradient descent which achieves an $\epsilon$-second-order critical point in $O(d / \epsilon^4)$ stochastic gradient queries, like PSGD.  This Riemannian version performs a certain number of steps in the tangent space, like PRGD.

More broadly, we anticipate that it should be possible to extend several classical optimization methods from the Euclidean case to the Riemannian case through this approach of running many steps in a given tangent space before retracting. This ought to be particularly beneficial for algorithms whose computations or analysis rely intimately on linear structures, such as for coordinate descent algorithms, certain parallelized schemes, and possibly also accelerated schemes. In preparing the final version of this paper, we found that this idea is also the subject of another paper at NeurIPS 2019, where it is called dynamic trivialization~\cite{lezcano2019trivialization}.


\subsubsection*{Acknowledgments}

We thank Yue Sun, Nicolas Flammarion and Maryam Fazel, authors of~\cite{Sun2019prgd}, for numerous relevant discussions. NB is partially supported by NSF grant DMS-1719558.


\bibliographystyle{plain}
\bibliography{ref}

\newpage

\begin{appendices}

\section{Proof that assumptions hold for compact manifolds}

\begin{proof}[Proof of Lemma~\ref{lem1}]
Since $\calM$ is compact and $f$ is continuous, $f$ is lower bounded by some $f^*$.

Recall $\hat{f}_x(s) = f \circ \Retr_x(s)$.  Define $\phi,\psi\colon\T\calM\rightarrow\reals$ using operator norms by 
$$\phi(x,s) = \norm{\nabla^2 \hat{f}_x(s)} = \norm{\nabla_s^2 (f \circ \Retr(x,s))},$$
$$\psi(x,s) = \norm{\nabla^3 \hat{f}_x(s)} = \norm{\nabla_s^3 (f \circ \Retr(x,s))}.$$
Since $f$ is three times continuously differentiable and $\Retr$ is smooth, $\phi$ and $\psi$ are each continuous on the tangent bundle $\T\calM$.
The set
\begin{align*}
	S_b & = \left\{ (x,s) : x \in\calM,s\in\T_x\calM \text{ with } \norm{s} \leq b \right\}
\end{align*}
is a compact subset of the tangent bundle $\T\calM$ since $\calM$ is compact. Thus, we may define
\begin{align*}
	L & = \max_{(x, s) \in S_b}\phi(x, s), & & \textrm{ and } & \rho & = \max_{(x, s) \in S_b}\psi(x, s),
\end{align*}
so that $\norm{\nabla^2 \hat{f}_x(s)} \leq L$ and $\norm{\nabla^3 \hat{f}_x(s)} \leq \rho$ for all $x\in\calM$ and $s\in B_{x,b}(0)$.  From here, it is clear that Assumptions \ref{as:A2} and \ref{as:A3} are satisfied, for we can just integrate as in eq.~\eqref{usefuleqn} below.

Using the notation from Assumption \ref{as:A4}, the map $\upsilon\colon\T\calM\rightarrow\reals$ given by $\upsilon(x,s) = \norm{\gamma_{x,s}''(0)}$ is continuous since $\Retr$ is smooth.  The set
\begin{align*}
	V_b & = \left\{ (x,s) : x \in\calM, s\in\T_x\calM \text{ with } \norm{s} = 1 \right\}
\end{align*}
is also compact in $\T\calM$. Hence, $\beta = \max_{(x, s) \in V_b}\upsilon(x, s)$ is a valid choice.
\end{proof}



\section{Proofs for the main results}

The proof follows that of Jin et al.~\cite{Jin2019} closely, reusing many of their key lemmas: we repeat some here for convenience, while highlighting the specificities of the manifold case. We consider it a contribution of this paper that, as a result of our distinction between manifold and tangent space steps, there is limited extra friction, despite the significantly extended generality. In this section and the next, all parameters are chosen as in~\eqref{params} and~\eqref{eq:FLscr}.
%
%

We assume $\epsilon \leq b^2 \rho$. We also assume $L \geq \sqrt{\rho\epsilon}$ because otherwise we can reach a point satisfying $\norm{\grad f(x)}\leq\epsilon$ and $\lambda_{\min}(\nabla^2 \hat{f}_x(0)) \geq - \sqrt{\rho\epsilon}$ simply using RGD. Indeed, RGD always finds a point $x\in\calM$ satisfying $\norm{\grad f(x)}\leq\epsilon$, and Assumption~\ref{as:A2} implies $\|\nabla^2\hat{f}_x(0)\| \leq L$ so that $\lambda_{\min}(\nabla^2\hat{f}_x(0)) \geq -L$.  Thus, if $L < \sqrt{\rho\epsilon}$, every point $x\in\calM$ satisfies $\lambda_{\min}(\nabla^2 \hat{f}_x(0)) \geq - \sqrt{\rho\epsilon}$.

We want to prove Theorem \ref{thmMaster}. This theorem follows from the following two lemmas (repeated from Lemmas~\ref{lemDog} and~\ref{lemCat} for convenience), which we prove in Appendix~\ref{next} below. Lemma~\ref{lem:manStep} is deterministic: it is a statement about the cost decrease produced by a single Riemannian gradient step, with bounded step size. Lemma~\ref{lem:lemma_11} is probabilistic, and is analogous to Lemma~11 in~\cite{Jin2019}.
\begin{lemma}\label{lem:manStep}
Under Assumptions~\ref{as:A2} and~\ref{as:A3}, set $\eta = 1/\ell$ for some $\ell \geq L$. If $x\in\calM$ satisfies $\norm{\grad f(x)} > \epsilon$ with $\epsilon \leq b^2 \rho$ and $L \geq \sqrt{\rho\epsilon}$, then, 
$$f(\textproc{TangentSpaceSteps}(x, 0, \eta, b, 1)) - f(x) \leq -\eta \epsilon^2/2.$$
\end{lemma}
\begin{lemma}\label{lem:lemma_11}
Under Assumptions~\ref{as:A2} and~\ref{as:A3}, let $x\in\calM$ satisfy both $\norm{\grad f(x)} \leq \epsilon$ and $\lambda_{\min}(\nabla^2 \hat{f}_x(0)) \leq -\sqrt{\rho\epsilon}$ with $\epsilon \leq b^2\rho$ and $L \geq \sqrt{\rho\epsilon}$. 
Set $\eta, r, \mathscr{T}, \mathscr{F}$ as in \eqref{params} and \eqref{eq:FLscr}.
Let $s_0 = \eta \xi$ with $\xi \sim \text{Uniform}(B_{x,r}(0))$.  Then,
\begin{align*}
	\mathbb{P}\big[f(\textproc{TangentSpaceSteps}(x, s_0, \eta, b, \mathscr{T})) - f(x) \leq -\mathscr{F}/2\big] \geq 1 - \frac{\ell\sqrt{d}}{\sqrt{\rho\epsilon}}2^{10-\chi/2}.
\end{align*}
\end{lemma}

\begin{proof}[Proof of Theorem \ref{thmMaster}]
This proof is similar to Jin et al.'s proof of Theorem 9 in~\cite{Jin2019}.

Recall that we set
\begin{align}
	T = 8\max\left\{\frac{\mathscr{T}}{3}, \frac{(f(x_0)-f^*)\mathscr{T}}{\mathscr{F}}, \frac{f(x_0)-f^*}{\eta \epsilon^2}\right\}.
	\label{eq:exactT}
\end{align}

PRGD performs two types of steps: (1) if $\norm{\grad f(x_t)} > \epsilon$, an RGD step on the manifold, and (2) if $\norm{\grad f(x_t)} \leq \epsilon$, a perturbation in the tangent space followed by GD steps in the tangent space.

There are at most $T/4$ iterates $x_t\in\calM$ satisfying $\norm{\grad f(x_t)} > \epsilon$ (i.e., iterates where an RGD step is performed), for otherwise Lemma~\ref{lem:manStep} and the definition of $T$~\eqref{eq:exactT} would imply $f(x_T) < f(x_0) - T\eta\epsilon^2/8 \leq f^*$, which contradicts Assumption~\ref{as:A1}.

The variable $t$ in Algorithm~\ref{PRGD} is an upper bound on the number of gradient queries issued so far.
For each RGD step on the manifold, $t$ increases by exactly 1.  PRGD does not terminate before $t$ exceeds $T$, and for every perturbation the counter increases by exactly $\mathscr{T}$.
Therefore, there are at least $3 T/(4\mathscr{T})$ iterates $x_t\in\calM$ satisfying $\norm{\grad f(x_t)} \leq \epsilon$. By the definition of $T$~\eqref{eq:exactT}, $3 T/(4\mathscr{T}) \geq 2$.

Suppose PRGD visits more than $T/(4 \mathscr{T})$ points $x_t\in \calM$ satisfying $\norm{\grad f(x_t)} \leq \epsilon$ and $\lambda_{\min}(\nabla^2 \hat{f}_{x_t}(0)) \leq -\sqrt{\rho\epsilon}$.  Each of these iterates $x_t$ is followed by a perturbation and at most $\mathscr{T}$ tangent space steps $\{s_j\}$.  For at least one such $x_t$, the sequence of tangent space steps does not escape the saddle point (that is, $f(x_{t+\mathscr{T}}) - f(x_t)> -\mathscr{F}/2$), for otherwise $f(x_T) < f(x_0) - T\mathscr{F}/(8\mathscr{T}) \leq f^*$ by the definition of $T$~\eqref{eq:exactT}.  Yet, by Lemma~\ref{lem:lemma_11} and a union bound, the probability that one or more of these sequences does not escape is at most $\delta$. Indeed, factoring out the third term in the max,
\begin{align*}
    T & = \frac{8\ell(f(x_0)-f^*)}{\epsilon^2} \max\bigg\{\frac{1}{3}\frac{\chi}{\sqrt{\rho\epsilon}}\frac{\epsilon^2}{(f(x_0)-f^*)}, 50\chi^4, 1\bigg\} \\
    & \leq \frac{8\ell(f(x_0)-f^*)}{\epsilon^2}\max\bigg\{\chi, 50\chi^4, 1\bigg\} = O \left(\frac{\ell(f(x_0) - f^*)}{\epsilon^2}\chi^4\right),
\end{align*}
where we used $\epsilon^{3/2} \leq 3\sqrt{\rho} \left( f(x_0) - f^* \right)$. Now using
\begin{align*}
	\max\left\{\chi, 50\chi^4, 1\right\} \leq 2^{18 + \chi/4}
\end{align*}
for all $\chi > 1/4$, and $\chi \geq 4 \log_2\left(2^{31} \frac{\ell^2\sqrt{d}(f(x_0)-f^*)}{\delta\sqrt{\rho}\epsilon^{5/2}}\right)$, we find
\begin{align*}
	T \cdot \frac{\ell\sqrt{d}}{\sqrt{\rho\epsilon}}2^{10-\chi/2} & \leq \frac{\ell^2\sqrt{d}}{\sqrt{\rho\epsilon}}\frac{(f(x_0)-f^*)}{\epsilon^2}2^{31-\chi/4} \leq \delta,
\end{align*}
as announced.

Hence, with probability at least $1-\delta$, PRGD visits at most $T/(4 \mathscr{T})$ points $x_t$ satisfying $\norm{\grad f(x_t)} \leq \epsilon$ and $\lambda_{\min}(\nabla^2 \hat{f}_{x_t}(0)) \leq -\sqrt{\rho\epsilon}$.  Using that there are at least $3 T/(4\mathscr{T})$ iterates $x_t\in\calM$ with $\norm{\grad f(x_t)}\leq\epsilon$, we conclude that at least two-thirds of the iterates $x_t\in\calM$ with $\norm{\grad f(x_t)}\leq\epsilon$ also satisfy $\lambda_{\min}(\nabla^2 \hat{f}_{x_t}(0)) \geq -\sqrt{\rho\epsilon}$, with probability at least $1-\delta$.
\end{proof}


Corollary \ref{thmMR} follows directly from Theorem \ref{thmMaster} and the following lemma.

\begin{lemma}
For some $\rho > 0$ (which would typically come from Assumption~\ref{as:A3}), under Assumption~\ref{as:A4} on the retraction,
let $x\in\calM$ satisfy $\norm{\grad f(x)}\leq \epsilon$ and $\lambda_{\min}(\nabla^2\hat{f}_x(0)) \geq -\sqrt{\rho\epsilon}$. Then,
$\lambda_{\min}(\Hess f(x)) \geq -\sqrt{\rho\epsilon} - \beta\epsilon$. In particular, if $\epsilon \leq \rho/\beta^2$, then $\lambda_{\min}(\Hess f(x)) \geq -\sqrt{4\rho\epsilon}$.
\end{lemma}
\begin{proof}
Considering $s = 0$ in Lemma~\ref{lems}, we may use $\Retr_x(0) = x$ and that $T_{x, 0}$ is the identity (as per Definition~\ref{def:retr}) to get $\nabla^2\hat{f}_x(0) = \Hess f(x) + W_0$, where
\begin{align*}
	\forall \dot s \in \T_x\calM \textrm{ with } \norm{\dot s} = 1, \qquad \inner{W_0[\dot s]}{\dot s} & \leq \norm{\gamma_{x, \dot s}''(0)} \norm{\grad f(x)} \leq \beta\epsilon.
\end{align*}
Thus, $\norm{W_0} \leq \beta\epsilon$ and
we find $\lambda_{\min}(\Hess f(x)) \geq -\sqrt{\rho\epsilon} - \beta\epsilon$. 
For the last part, use $\beta \leq \sqrt{\rho/\epsilon}$.
\end{proof}

Corollary \ref{compactCor} follows directly from Corollary \ref{thmMR} and Lemma \ref{lem1}.

\section{Proofs of key lemmas}\label{next}
The goal of this section is to prove Lemmas \ref{lem:manStep} and \ref{lem:lemma_11}.  All proofs deal with linear spaces, not manifolds. The key ideas are due to Jin et al.~\cite{Jin2019}.
The following lemma is needed because to apply Jin et al.'s analysis we need the pullbacks not only to satisfy the restricted Lipschitz condition, Assumption~\ref{as:A2}, but also to have Lipschitz continuous gradient at least, uniformly in tangent space balls of fixed radius. The lemma below implies Lemma~\ref{lem:Aell}.
\begin{lemma}\label{lem:rad_a}
Let $f$ satisfy Assumptions \ref{as:A2} and \ref{as:A3}, and let $\ell = L + \rho b$. For all $x \in \calM$, it holds that $\nabla \hat{f}_x$ is $\ell$-Lipschitz continuous in the ball $B_{x,b}(0) \subset \T_x \calM$.
\end{lemma}
\begin{proof} 
By Assumption~\ref{as:A2}, $\norm{\nabla^2\hat{f}_x(0)} \leq L$.  Hence, by Assumption~\ref{as:A3}, for all $s \in B_{x,b}(0)$,
\begin{align*}
	\norm{\nabla^2\hat{f}_x(s)} \leq \norm{\nabla^2\hat{f}_x(0)} + \norm{\nabla^2\hat{f}_x(s) - \nabla^2\hat{f}_x(0)} \leq L + \rho \norm{s} \leq L + \rho b = \ell.
\end{align*}
Let $s_1, s_2 \in B_{x,b}(0)$ be arbitrary.  Then indeed,
\begin{equation}\label{usefuleqn}
\norm{\nabla \hat{f}_x(s_2) - \nabla \hat{f}_x(s_1)} = \norm{\int_0^1 \nabla^2 \hat{f}_x(s_1 + (s_2-s_1)\tau)[s_2-s_1] d\tau} \leq \ell \norm{s_2-s_1},
\end{equation}
where we used that
the line segment from $s_1$ to $s_2$ is contained in $B_{x,b}(0)$.
\end{proof}

Together with the one above, the following standard lemma allows us to establish the sufficient decrease of $\hat{f}_x$ in $B_{x,b}(0)$ upon taking a gradient step in the tangent space.

\begin{lemma}\label{lem:lemma_10'}
%
Let $\nabla \hat{f}_x$ be $\ell$-Lipschitz continuous along the line segment connecting $s_j$ to $s_{j+1}$,
related by $s_{j+1} = s_j - \alpha \eta \nabla \hat{f}_{x}(s_j)$ with $\eta = 1/\ell$ and $\alpha \in [0, 1]$. Then,
\begin{align*}
	\hat{f}_{x}(s_{j+1}) - \hat{f}_{x}(s_{j}) \leq - \frac{\alpha\eta}{2}\sqnorm{\nabla \hat{f}_{x}(s_j)}.
\end{align*}
\end{lemma}
\begin{proof} 
	It is a standard consequence of Lipschitz continuity of $\nabla \hat f_x$ along the line segment $\tau \mapsto (1-\tau)s_j + \tau s_{j+1}$ for $\tau \in [0, 1]$ that
	\begin{align*}
		\hat f_x(s_{j+1}) \leq \hat f_x(s_j) + \inner{\nabla \hat f_x(s_j)}{s_{j+1} - s_j} + \frac{\ell}{2} \norm{s_{j+1} - s_j}^2.
	\end{align*}
	Plugging in $s_{j+1} - s_j = - \alpha \eta \nabla \hat{f}_{x}(s_j)$, we get
	\begin{align*}
		\hat f_x(s_{j+1}) \leq \hat f_x(s_j) + \left[ - \alpha\eta + \frac{\ell \alpha^2 \eta^2}{2} \right] \norm{\nabla \hat f_x(s_j)}^2.
	\end{align*}
	The coefficient between brackets is further equal to $\left(-1 + \frac{\alpha}{2}\right)\alpha\eta$, which is at most $-\alpha\eta/2$.
\end{proof}

We are now ready to prove Lemma \ref{lem:manStep}.
\begin{proof}[Proof of Lemma \ref{lem:manStep}]
	The call to $\textproc{TangentSpaceSteps}(x, 0, \eta, b, 1)$ produces a point $\Retr_x(s_1)$, with $s_1 = s_0 - \alpha\eta\nabla\hat f_x(s_0)$, where $s_0 = 0$, $\alpha \in [0, 1]$ and $\norm{s_1} \leq b$. Owing to Assumption~\ref{as:A2}, we know that $\nabla \hat f_x$ is $L$-Lipschitz continuous along the line segment connecting $s_0$ to $s_1$. Since $\ell \geq L$, it is a fortiori $\ell$-Lipschitz continuous along that line segment: Lemma~\ref{lem:lemma_10'} applies and yields
	\begin{align*}
		f(\Retr_x(s_1)) = \hat{f}_{x}(s_{1}) & \leq \hat{f}_{x}(s_{0}) - \frac{\alpha\eta}{2}\sqnorm{\nabla \hat{f}_{x}(s_0)} = f(x) - \frac{\alpha \eta}{2} \norm{\nabla \hat f_x(0)}^2.
	\end{align*}
	If $\alpha = 1$, since $\norm{\nabla \hat f_x(0)} = \norm{\grad f(x)} > \epsilon$, we are done. Owing to how $\textproc{TangentSpaceSteps}$ works, if $\alpha < 1$, then it must be that $\|\alpha \eta \nabla \hat f_x(0)\| = b$, so that the inequality above yields
	\begin{align*}
		f(\Retr_x(s_1))  & \leq f(x) - \frac{b}{2}\norm{\nabla \hat f_x(0)} \leq f(x) - \frac{b\epsilon}{2}.
	\end{align*}
	Using $\epsilon \leq b^2\rho$ and $\ell \geq L\geq \sqrt{\rho\epsilon}$,
	\begin{align*}
		\eta \epsilon = \frac{\epsilon}{\ell} \leq \frac{\sqrt{b^2\rho\epsilon}}{\ell} = \frac{\sqrt{\rho\epsilon}}{\ell}b \leq b.
	\end{align*}		
	Hence, $f(\Retr_x(s_1)) \leq f(x) - \eta\epsilon^2/2$, as desired.
%
	(As a side note: Assumption~\ref{as:A3} is not truly necessary here; it is only convenient so that we can use the same definitions of $\rho, b$ and $\ell$ as in other parts of the paper.)
%
\end{proof}

Lemma \ref{lem:lemma_12} is Jin et al.'s ``improve or localize lemma''~\cite{Jin2019}, with a tweak for variable step sizes.  The lemma states that if the function value does not decrease much, then the iterates are localized.

\begin{lemma}\label{lem:lemma_12}
Fix $j \geq 0$, $x \in \calM$ and $s_0 \in \T_{x} \calM$.  For all $0 \leq i \leq j-1$, assume $0 \leq \eta_i \leq \eta = 1/\ell$, $s_{i+1} = s_i - \eta_i \nabla \hat{f}_{x}(s_i)$ and $\nabla\hat{f}_x$ is $\ell$-Lipschitz continuous along the line segment connecting $s_i$ to $s_{i+1}$. Then,
$$\norm{s_{j} - s_0} \leq \sqrt{2\eta j\big(\hat{f}_{x}(s_0) - \hat{f}_{x}(s_j)\big)}.$$
\end{lemma}
\begin{proof} 
Using a telescoping sum, triangle inequality, Cauchy--Schwarz and (to get to the last line) Lemma \ref{lem:lemma_10'}, we get:
\begin{align*}
	\norm{s_{j} - s_0} & = \norm{\sum_{i = 0}^{j-1} s_{i+1}-s_{i}} = \norm{\sum_{i = 0}^{j-1} - \eta_{i} \nabla \hat{f}_{x}(s_{i})} \leq \sum_{i = 0}^{j-1} \sqrt{\eta_{i}}\norm{\sqrt{\eta_{i}} \nabla \hat{f}_{x}(s_{i})} \\ 
					  & \leq \sqrt{\bigg(\sum_{i = 0}^{j-1} \eta_i \sqnorm{\nabla \hat{f}_{x}(s_{i})}\bigg)\bigg(\sum_{i = 0}^{j-1} \eta_{i}\bigg)} \leq \sqrt{2\eta j\bigg(\sum_{i = 0}^{j-1} \frac{\eta_{i}}{2} \sqnorm{\nabla\hat{f}_{x}(s_{i})}\bigg)} \\
					  & \leq \sqrt{2 \eta j \bigg(\sum_{i = 0}^{j-1} \hat{f}_{x}(s_{i}) - \hat{f}_{x}(s_{i+1})\bigg)} = \sqrt{2 \eta j \big(\hat{f}_{x}(s_{0}) - \hat{f}_{x}(s_{j})\big)}. \qedhere
%
\end{align*}
\end{proof}

Lemma \ref{lem:lemma_13} below and its proof are very similar to Jin et al.'s Lemma 13 and its proof~\cite{Jin2019}, except for a modification since $\nabla\hat{f}_x$ is only Lipschitz continuous in a ball. This deterministic lemma formalizes the coupling sequence argument: if the Hessian of the pullback has a negative eigenvalue which is large in magnitude, upon initializing the tangent space steps at two appropriately chosen points $s_0^{}, s_0'$, with certainty, one of them leads to significant decrease in the cost function.
As usual, we use parameters $\eta, r, \mathscr{T}$ as in~\eqref{params} and $\mathscr{F}, \mathscr{L}$ as in~\eqref{eq:FLscr}.
\begin{lemma}\label{lem:lemma_13}
Under Assumptions~\ref{as:A2} and~\ref{as:A3}, let $x \in \calM$ be such that $\lambda_{\min}(\nabla^2 \hat{f}_x(0)) \leq - \sqrt{\rho\epsilon}$, with $\epsilon \leq b^2\rho$ and $L \geq \sqrt{\rho\epsilon}$.
Let $s_0^{}, s_0' \in \T_{x}\calM$ be such that
\begin{enumerate}
\item $\norm{s_0}, \norm{s_0'} \leq \eta r$, and
\item $s_0^{} - s_0' = \eta r_0 e_1$, where $e_1$ is an eigenvector of unit norm associated with the minimum eigenvalue of $\nabla^2 \hat{f}_x(0)$, and $r_0 > \omega = 2^{2-\chi}\ell \mathscr{L}$.
\end{enumerate}
Let $s_{\mathscr{T}}$ be defined by running $\textproc{TangentSpaceSteps}(x, s_0, \eta, b, \mathscr{T})$ (see Algorithm~\ref{PRGD}).   Let $s_{\mathscr{T}}'$ be similarly defined by running $\textproc{TangentSpaceSteps}(x, s_0', \eta, b, \mathscr{T})$.  Then,
$$\min\left\{\hat{f}_{x}(s_{\mathscr{T}})-\hat{f}_{x}(s_{0}), \hat{f}_{x}(s_{\mathscr{T}}') - \hat{f}_{x}(s_{0}')\right\} \leq - \mathscr{F}.$$
\end{lemma}

\begin{proof} 
First, note that both sequences are initialized in the interior of the ball of radius $b$. Indeed, using $\ell \geq L, L \geq \sqrt{\rho \epsilon}$, $\epsilon \leq b^2\rho$ and $\chi > 1/4$,
\begin{align}
\eta r = \frac{1}{\ell}\frac{\epsilon}{400 \chi^3} < \frac{\epsilon}{L}\frac{64}{400} = b \sqrt{\frac{\rho\epsilon}{L^2}\frac{\epsilon}{b^2\rho}}\frac{64}{100} \leq b\frac{64}{100} < b.
\label{eq:etarlessb}
\end{align}
	
The proof is by contradiction: assume
\begin{align*}
	\min\left\{\hat{f}_{x}(s_{\mathscr{T}})-\hat{f}_{x}(s_{0}), \hat{f}_{x}(s_{\mathscr{T}}') - \hat{f}_{x}(s_{0}')\right\} > - \mathscr{F}.
\end{align*}
Further assume, for the sake of contradiction, that one of the sequences $\{s_j\}_{j \leq \mathscr{T}}, \{s_j'\}_{j \leq \mathscr{T}}$ (defined in $\textproc{TangentSpaceSteps}$) escapes the interior of the ball $B_{x,b}(0)$.  
Without loss of generality, assume $\{s_j\}_{j\leq\mathscr{T}}$ escapes.  Let $j \leq \mathscr{T}-1$ be the minimum integer for which $\norm{s_{j+1}} \geq b$.  Then, $\textproc{TangentSpaceSteps}(x, s_0, \eta, b, \mathscr{T})$ terminates with $s_j - \alpha \eta \nabla \hat{f}_{x}(s_j)$ for some $\alpha \in (0,1]$ satisfying $b = \norm{s_j - \alpha \eta \nabla \hat{f}_{x}(s_j)}$.  Using Lemma \ref{lem:lemma_12}, $\ell \geq L \geq \sqrt{\rho\epsilon}$ and  $\chi > \frac{1}{4}$,
\begin{align*}
b & = \norm{s_j - \alpha \eta \nabla \hat{f}_{x}(s_j)} \leq \norm{s_j - \alpha \eta \nabla \hat{f}_{x}(s_j)-s_0}+\norm{s_0} \leq \sqrt{2\eta(j+1)\mathscr{F}} + \eta r \\ &\leq \sqrt{2\eta\mathscr{T}\mathscr{F}} + \eta r \leq \sqrt{\frac{\epsilon}{25\chi^2\rho}} + \frac{1}{400 \chi^3}\sqrt{\frac{\epsilon}{\rho}} \leq \bigg(\frac{1}{5\chi} + \frac{1}{400 \chi^3} \bigg)\sqrt{\frac{\epsilon}{\rho}} \leq  \frac{1}{4\chi}\sqrt{\frac{\epsilon}{\rho}} = \mathscr{L}.
\end{align*}
Since $\epsilon \leq b^2\rho$, we know that $\mathscr{L} < b$, which shows a contradiction. Thus, neither of the sequences $\{s_j\}_{j \leq \mathscr{T}}, \{s_j'\}_{j \leq \mathscr{T}}$ leave the interior of $B_{x,b}(0)$.  That is, $s_{j+1} = s_j - \eta \nabla \hat{f}_{x}(s_j)$ and $\norm{s_{j+1}} < b$ for $j = 0, 1, \ldots, \mathscr{T}-1$, and similarly for $\{s_j'\}_{j \leq \mathscr{T}}$.

From here, we proceed exactly as in Lemma 13 of~\cite{Jin2019}.  By Lemma \ref{lem:lemma_12}, for all $j \leq \mathscr{T}$,
\begin{equation}\label{that}
\max\left\{\norm{s_j}, \norm{s_j'}\right\} \leq \max\left\{\norm{s_j-s_0}, \norm{s_j'-s_0'}\right\} + \eta r \leq \sqrt{2\eta\mathscr{T}\mathscr{F}} + \eta r \leq \mathscr{L}.
\end{equation}
Let $\hat{s}_j = s_j - s_j'$ and $\mathcal{H} = \nabla^2 \hat{f}_{x}(0)$.  Then,
\begin{align*}
\hat{s}_{j+1} & = \hat{s}_j - \left(\eta \nabla \hat{f}_{x}(s_j) - \eta \nabla \hat{f}_{x}(s_j')\right) = \hat{s}_j - \eta \int_0^1 {\nabla^2 \hat{f}_{x}\left(s_j' + \theta(s_j - s_j')\right) [s_j-s_j']} d\theta \\ & = (I- \eta \mathcal{H})\hat{s}_j - \eta \Delta_j \hat{s}_j,
\end{align*}
where $\Delta_j = \int_0^1 {\left(\nabla^2 \hat{f}_{x}\left(s_j' + \theta(s_j - s_j')\right) - \mathcal{H}\right)} d\theta$.  By Assumption \ref{as:A3},
\begin{align*}
\norm{\Delta_j} \leq \int_0^1 \rho\norm{s_j' + \theta(s_j - s_j')} d\theta \leq \int_0^1 \rho \max\{\norm{s_j}, \norm{s_j'}\}d\theta \leq \rho \mathscr{L}.
\end{align*}
This will be useful momentarily. It is easy to check by induction that
\begin{align*}
	\hat{s}_{j+1} & = p(j+1) - q(j+1),
\end{align*}
where $p(0) = \hat s_0, q(0) = 0$, and
\begin{align*}
	p(j+1) & = (I - \eta \mathcal{H})^{j+1}\hat{s}_0, & & \text{ and } & q(j+1) & = \eta \sum_{i=0}^j(I-\eta \mathcal{H})^{j-i}\Delta_{i}\hat{s}_{i}.
\end{align*}

We use induction to show that $\norm{q(j)} \leq \norm{p(j)}/2$.  The claim is clearly true for $j=0$.  Suppose the claim is true for all $i \leq j$.  We prove the claim for $j+1$.  Let $-\gamma = \lambda_{\min}(\nabla^2 \hat{f}_x(0))$. Using $\hat{s}_0 = \eta r_0 e_1$, notice in particular that
\begin{align*}
	p(j) & = (I - \eta \mathcal{H})^j \eta r_0 e_1 = (1+\eta \gamma)^j \eta r_0 e_1,
\end{align*}
so that the norm of $p(j)$ grows with $j$: $\norm{p(j)} = (1+\eta \gamma)^j \eta r_0$.
Using the induction hypothesis, for all $i \leq j$ we have:
$$\norm{\hat s_i} \leq \norm{p(i)} + \norm{q(i)} \leq \frac{3}{2}\norm{p(i)} \leq 2(1+\eta \gamma)^i \eta r_0.$$
Furthermore, since $\mathcal{H} \preceq L I \preceq \ell I$, it follows that $I - \eta\mathcal{H} \succeq 0$. As a result, $\norm{I - \eta\mathcal{H}} = \lambda_{\max}(I - \eta\mathcal{H}) = 1+\eta\gamma$.
Therefore, also using $2\eta\rho\mathscr{L}\mathscr{T} = 1/2$ in the last step,
\begin{align*}
    \norm{q(j+1)} &= \norm{\eta \sum_{i=0}^j(I-\eta\mathcal{H})^{j-i}\Delta_i\hat{s}_i} \leq \eta \rho\mathscr{L}\sum_{i=0}^j\norm{(I-\eta\mathcal{H})^{j-i}}\norm{\hat{s}_i} \\ & \leq 2\eta\rho\mathscr{L}\sum_{i=0}^j(1+\eta\gamma)^{j-i}(1+\eta\gamma)^i\eta r_0  \\ & \leq 2\eta\rho\mathscr{L}\mathscr{T}(1+\eta\gamma)^j\eta r_0 = 2\eta\rho\mathscr{L}\mathscr{T}\norm{p(j)} \leq \norm{p(j+1)}/2,
\end{align*}
So we have proven $\norm{q(j)} \leq \norm{p(j)}/2$ for all $j$.  Therefore, using the definition of $r_0$ in the last step,
\begin{align*}
    \max\{\norm{s_{\mathscr{T}}},\norm{s'_{\mathscr{T}}}\} & \geq (\norm{s_{\mathscr{T}}}+\norm{s_{\mathscr{T}}'})/2 \geq \norm{\hat{s}_{\mathscr{T}}}/2 \geq (\norm{p(\mathscr{T})} - \norm{q(\mathscr{T})})/2 \\ &\geq \norm{p(\mathscr{T})}/4 = (1+\eta\gamma)^{\mathscr{T}}\eta r_0 / 4 \geq 2^{\chi-2}\eta r_0 > \mathscr{L},
\end{align*}
which contradicts~\eqref{that}. In the second to last step, we used $\gamma \geq \sqrt{\rho \epsilon}$ and $\sqrt{\rho \epsilon} \leq \ell$ so that
\begin{align*}
	\frac{1}{\chi}\log_2\left((1+\eta\gamma)^{\mathscr{T}} \right) \geq \frac{\mathscr{T}}{\chi} \log_2\left(1 + \frac{\sqrt{\rho \epsilon}}{\ell} \right) =\frac{\mathscr{T}}{\chi} \log_2\left(1 + \frac{\chi}{\mathscr{T}} \right) \geq 1,
\end{align*}
since $\frac{1}{\alpha} \log_2(1+\alpha) \geq 1$ for all $\alpha \in [0, 1]$.
Except for the initial part, this proof is due to Jin et al.~\cite{Jin2019}.
\end{proof}

We are now ready to prove Lemma \ref{lem:lemma_11}.  This proof is completely due to Jin et al.~\cite{Jin2019}: we only somewhat modify how the proof is presented.

\begin{proof}[Proof of Lemma~\ref{lem:lemma_11}]
Recall that $\eta r < b$~\eqref{eq:etarlessb}, and define the stuck region
$$\mathcal{X}_{\mathrm{stuck}} = \big\{s \in B_{x,\eta r}(0) : f(\textproc{TangentSpaceSteps}(x,s,\eta,b,\mathscr{T})) - f(x) > -\mathscr{F}\big\}.$$
Running the tangent space steps with $s_0$ in that set does not yield sufficient improvement of the cost function despite the fact that the Hessian has a negative eigenvalue with large magnitude, hence the name. We aim to show that this set has a small volume, so that it is unlikely to encounter it by random chance.

Let $S_{e_1}$ be the subspace of $\T_x\calM$ orthogonal to $e_1$.
Given $a \in S_{e_1} \cap B_{x,\eta r}(0)$, let $\ell_a$ denote the line in $\T_x\calM$ parallel to $e_1$ passing through $a$.  Then, with $\mathbbm{1}$ denoting the indicator function,
$$\text{Vol}(\mathcal{X}_{\mathrm{stuck}}) = \int_{\T_x\calM}\mathbbm{1}_{\mathcal{X}_{\mathrm{stuck}}}(y)dy = \int_{S_{e_1} \cap B_{x,\eta r}(0)} \left[ \int_{\ell_a}\mathbbm{1}_{\mathcal{X}_{\mathrm{stuck}}}(z) dz \right] da.$$
Lemma~\ref{lem:lemma_13} states any two points that are both on the line $\ell_a$ \emph{and} in $\mathcal{X}_{\mathrm{stuck}}$ must be close. Specifically, for all $s, s' \in \ell_a\cap\mathcal{X}_{\mathrm{stuck}}$, we have $\norm{s-s'} \leq \eta \omega$, with $\omega = 2^{2-\chi}\ell \mathscr{L}$. Therefore, the set of problematic points on the line $\ell_a$ is contained in a segment of length at most $\eta \omega$ and we deduce $\int_{\ell_a}\mathbbm{1}_{\mathcal{X}_{\mathrm{stuck}}}(z) dz \leq \eta \omega$. As a result,
$$\text{Vol}(\mathcal{X}_{\mathrm{stuck}}) \leq \eta\omega\int_{S_{e_1} \cap B_{x,\eta r}(0)} da = \eta\omega \text{Vol}(\mathbb{B}^{d-1}_{\eta r}),$$
where $\mathbb{B}^k_R$ denotes a $k$-dimensional (Euclidean) ball of radius $R$.  Since $s_0 \sim \text{Uniform}(B_{x,\eta r}(0))$,
\begin{align*}
    \mathbb{P}(s_0 \in \mathcal{X}_{\mathrm{stuck}}) &= \frac{\text{Vol}(\mathcal{X}_{\mathrm{stuck}})}{\text{Vol}(\mathbb{B}^d_{\eta r})} \leq \frac{\eta\omega \text{Vol}(\mathbb{B}^{d-1}_{\eta r})}{\text{Vol}(\mathbb{B}^d_{\eta r})} =\frac{\omega \Gamma(1+d/2)}{r\sqrt{\pi}\Gamma((d+1)/2)} \leq \frac{\omega}{r}\sqrt{\frac{d}{\pi}}\ =\frac{\ell\sqrt{d}}{\sqrt{\rho \epsilon}}\frac{400}{\sqrt{\pi}}2^{-\chi}\chi^2 \\ &\leq \frac{\ell\sqrt{d}}{\sqrt{\rho \epsilon}}2^{10-\chi/2},
\end{align*}
where we used the Gautschi inequality for the $\Gamma$ function, and $\chi > 1/4$ to bound $\frac{400}{\sqrt{\pi}}2^{-\chi}\chi^2 \leq 2^{10-\chi/2}$.
To conclude, note that if $s_0 \not\in\mathcal{X}_{\mathrm{stuck}}$ then
\begin{multline*}
    f(\textproc{TangentSpaceSteps}(x,s_0,\eta,b,\mathscr{T})) - f(x) \\ = f(\textproc{TangentSpaceSteps}(x,s_0,\eta,b,\mathscr{T})) - \hat{f}_x(s_0) + \hat{f}_x(s_0) - \hat{f}_x(0) \\ \leq -\mathscr{F} + \epsilon \eta r + \ell \eta^2r^2/2 = -\mathscr{F} + \frac{\sqrt{\rho\epsilon}}{\ell}\mathscr{F}\big(50\chi^3\big)\bigg(\frac{1}{400\chi^3} + \frac{1}{2}\Big(\frac{1}{400\chi^3}\Big)^2\bigg) \\ \leq -\mathscr{F} + \mathscr{F}\bigg(\frac{1}{8} + \frac{25}{400^2\chi^3}\bigg) \leq -\mathscr{F}/2,
\end{multline*}
using $\sqrt{\rho \epsilon} \leq \ell$ and $\chi > 1/4$ once more in the last step, and also
\begin{align*}
	\hat{f}_x(s_0) - \hat{f}_x(0) \leq \epsilon \eta r + \ell \eta^2r^2/2
\end{align*}
owing to the fact that $\nabla \hat f_x$ is $\ell$-Lipschitz continuous along the line segment connecting 0 and $s_0$, $\|s_0\| \leq \eta r$ and $\|\nabla \hat f_x(0)\| \leq \epsilon$.
\end{proof}

\section{Regularity constants for dominant eigenvector computation (PCA)} \label{sec:PCA}

Computing the dominant eigenvector of a symmetric matrix $A \in \Rnn$ (which notably comes up in PCA) comes down to solving
\begin{align}
	\max_{x \in \Sn} f(x),  & & f(x) = \frac{1}{2} x\transpose A x,
\end{align}
where $\Sn = \{ x \in \Rn : x\transpose x = 1 \}$ is the unit sphere. If we use the retraction $\Retr_x(s) = \frac{x+s}{\|x+s\|}$ (where $\|x\| = \sqrt{x\transpose x}$)---for which Assumption~\ref{as:A4} holds with $\beta = 0$---then pullbacks are of the form
\begin{align}
	\hat f_x(s) & = f(\Retr_x(s)) = \frac{1}{1+\|s\|^2} \frac{1}{2} (x+s)\transpose A (x+s),
\end{align}
defined over the tangent spaces $\T_x\Sn = \{ s \in \Rn : x\transpose s = 0 \}$. The gradient of $\hat f_x$ at $s$ is given by
\begin{align}
	\nabla \hat f_x(s) & = \Proj_x\!\left( \frac{1}{1+\|s\|^2} A(x+s) + \frac{-1}{(1+\|s\|^2)^2} (x+s)\transpose A (x+s) \cdot s \right) \\
					   & = \frac{1}{1+\|s\|^2} \left( \Proj_x(A(x+s)) - 2\hat f_x(s) \cdot s \right),
\end{align}
where $\Proj_x(s) = s - (x\transpose s) x$ is the orthogonal projector from $\Rn$ to $\T_x\Sn$. It follows that
\begin{align}
	\nabla \hat f_x(s) - \nabla \hat f_x(0) & = \frac{1}{1+\|s\|^2} \left( \Proj_x(As) - 2 \hat f_x(s) s - \|s\|^2 \Proj_x(Ax) \right).
\end{align}
Using $\frac{1}{1+\|s\|^2} \leq 1$ and $\frac{\|s\|^2}{1+\|s\|^2} \leq \frac{1}{2}\|s\|$ for all $s$, and using the fact that an orthogonal projector can only reduce the norm of a vector, we find
\begin{align}
	\|\nabla \hat f_x(s) - \nabla \hat f_x(0)\| & \leq \|As\| + 2 \left[\sup_{s\in\T_x\Sn} |\hat f_x(s)| \right] \|s\| + \frac{1}{2} \|Ax\| \|s\|.
\end{align}
Letting $\|A\|$ denote the operator norm of $A$ (largest singular value), we finally obtain
\begin{align}
	\|\nabla \hat f_x(s) - \nabla \hat f_x(0)\| & \leq \frac{5}{2} \|A\| \|s\|.
\end{align}
This shows that Assumption~\ref{as:A2} holds with $b_1 = \infty$ and $L = \frac{5}{2}\|A\|$, or any larger number. For example, the induced 1-norm of the matrix $A$ is straightforward to compute and is an upper-bound on $\|A\|$.

Now aiming to control second-order derivatives, we compute a directional derivative of $\nabla \hat f_x(s)$ and obtain the Hessian of $\hat f_x$ on the tangent space $\T_x\Sn$:
\begin{align*}
	\nabla^2 \hat f_x(s)[\dot s] & = -2 \frac{\inner{s}{\dot s}}{1+\|s\|^2} \nabla \hat f_x(s) + \frac{1}{1+\|s\|^2} \left[ \Proj_x(A\dot s) - 2 \hat f_x(s) \dot s - 2\innersmall{\nabla \hat f_x(s)}{\dot s} s \right],
\end{align*}
where $\inner{u}{v} = u\transpose v$. In particular, $\nabla^2 \hat f_x(0)[\dot s] = \Proj_x(A\dot s) - (x\transpose A x) \dot s$, so that
\begin{multline}
	\inner{\dot s}{\left( \nabla^2 \hat f_x(s) - \nabla^2 \hat f_x(0) \right)[\dot s]} = -4 \frac{\inner{s}{\dot s} \innersmall{\nabla \hat f_x(s)}{\dot s}}{1+\|s\|^2} + \left( \frac{1}{1+\|s\|^2} - 1 \right) \innersmall{\dot s}{A \dot s} \\ - \left( 2 \frac{\hat f_x(s)}{1+\|s\|^2} - x\transpose A x \right) \|\dot s\|^2.
\end{multline}
Using $\frac{\|s\|}{1+\|s\|^2} \leq \frac{1}{2}$, it is easy to see that $\|\nabla \hat f_x(s)\| \leq \frac{3}{2}\|A\|$ and:
\begin{align*}
	\|\nabla^2 \hat f_x(s) - \nabla^2 \hat f_x(0)\| & \leq 4 \|\nabla \hat f_x(s)\| \|s\| + \frac{1}{2} \|A\| \|s\| \\ & \quad + \left[ \sup_{s\in\T_x\Sn} \frac{\left| (x+s)\transpose A (x+s) - (1+\|s\|^2)^2 x\transpose A x \right|}{(1+\|s\|^2)^2 \|s\|} \right] \|s\| \\
		& \leq (6+1/2)\|A\| \|s\| + \left[ \sup_{t > 0} \frac{2t + 3t^2 + t^4}{(1+t^2)^2 t} \right] \|A\| \|s\| \\
		& \leq 9 \|A\| \|s\|.
\end{align*}
This shows Assumption~\ref{as:A3} holds with $b_2 = \infty$ and $\rho = 9\|A\|$.

\end{appendices}

\end{document}

%% file: preamble.tex
\usepackage{amsmath}
\usepackage{mathtools}
\usepackage{bbm}

\usepackage{amsfonts}
\usepackage{amsthm}  
\usepackage{amssymb}
\usepackage{color}
\usepackage{bm} 
\usepackage{dsfont} 
\usepackage{graphicx}
\usepackage{algorithm}
\usepackage{algpseudocode}
\algdef{SE}[DOWHILE]{Do}{doWhile}{\algorithmicdo}[1]{\algorithmicwhile\ #1}%
\usepackage{mathrsfs}

\newcommand{\transpose}{^\top\! }

\newcommand{\inner}[2]{\left\langle{#1},{#2}\right\rangle}
\newcommand{\innersmall}[2]{\langle{#1},{#2}\rangle}

\newcommand{\Proj}{\mathrm{Proj}}

\newcommand{\Retr}{\mathrm{Retr}}

\newcommand{\T}{\mathrm{T}}

\newcommand{\Sn}{{{S}^{n-1}}} 

\newcommand{\Rnn}{{\mathbb{R}^{n\times n}}}

\newcommand{\Rd}{{\mathbb{R}^{d}}}

\newcommand{\reals}{{\mathbb{R}}}

\newcommand{\Rn}{{\mathbb{R}^n}}

\newcommand{\grad}{\mathrm{grad}\,}
\newcommand{\Hess}{\mathrm{Hess}\,}

\newcommand{\D}{\mathrm{D}}

\newcommand{\calM}{\mathcal{M}}

\newcommand{\norm}[1]{\left\|{#1}\right\|}
\newcommand{\sqnorm}[1]{\left\|{#1}\right\|^2}
\newcommand{\TODOinvisible}[1]{}

\newcommand{\dist}{\mathrm{dist}}

\newtheorem{theorem}{Theorem}[section] 
\newtheorem{lemma}[theorem]{Lemma} 
\newtheorem{corollary}[theorem]{Corollary} 
\newtheorem{assumption}{Assumption}
\newtheorem{definition}[theorem]{Definition} 
\newtheorem{remark}[theorem]{Remark}

\usepackage{algorithm}
\usepackage{mathrsfs}

\makeatletter
\def\BState{\State\hskip-\ALG@thistlm}
\makeatother

\usepackage[toc,page]{appendix}
\RequirePackage[colorlinks,citecolor=blue,urlcolor=blue,linkcolor=blue]{hyperref}